\documentclass[a4paper,10pt]{amsart}

\usepackage{amsthm}
\usepackage[all]{xy}

\newtheorem*{theorem}{Theorem}
\newtheorem*{conjecture}{Conjecture}
\newtheorem*{proposition}{Proposition}
\newtheorem*{question}{Question}
\theoremstyle{definition}
\newtheorem*{definition}{Definition}

\begin{document}

\begin{abstract}
This short note is an extended abstract of a talk given at the conference \lq\lq Komplexe Analysis\rq\rq{} at the Mathematisches Forschungsinstitut Oberwolfach in September 2012. We explained some recent results about the existence of rational curves on Calabi-Yau threefolds as well as a curvature approach to the non hyperbolicity of such manifolds.
\end{abstract}

\author{Simone Diverio}
\title[Rational curves on Calabi-Yau threefolds]{Rational curves on Calabi-Yau threefolds and a conjecture of Oguiso}
\address{Simone Diverio \\ CNRS et Institut de Math\'ematiques de Jussieu -- Paris Rive Gauche.}
\email{diverio@math.jussieu.fr} 
\thanks{The author is partially supported by the ANR project \lq\lq POSITIVE\rq\rq{}, ANR-2010-BLAN-0119-01.}
\maketitle

Let $X$ be a compact projective manifold over $\mathbb C$, $\omega$ a K\"ahler metric on $X$, and consider the following statements:

\begin{itemize}
\item[(1)] The holomorphic sectional curvature of $\omega$ is strictly negative.
\item[(2)] The manifold $X$ has non-degenerate negative $k$-jet curvature.
\item[(3)] The manifold $X$ is Kobayashi hyperbolic.
\item[(4)] The manifold $X$ is measure hyperbolic.
\item[(5)] The manifold $X$ is of general type.
\item[(6)] The manifold $X$ does not contain any rational curve.
\item[(7)] The canonical bundle $K_X$ of $X$ is nef.
\item[(8)] The canonical bundle $K_X$ of $X$ is ample.
\end{itemize}

Perhaps only $(2)$ needs some more explanations, which will be given subsequently. We have the following diagram of conjectural and actual implications:

\smallskip
$$
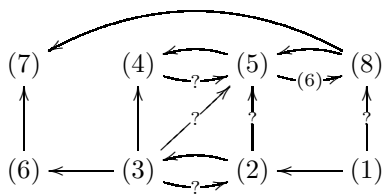
\xymatrix{
(7) & (4) \ar@/_/[r]|?  & (5)\ar@/_/[l] \ar@/_/[r]|{(6)}  & (8) \ar@/_/[l] \ar@/_1.7pc/[lll] \\
(6) \ar[u] & (3) \ar[u] \ar[l] \ar[ur]|? \ar@/_/[r]|? & (2) \ar@/_/[l] \ar[u]|? & (1) \ar[l] \ar[u]|?}
$$

As the diagram shows, the central conjecture here is the equivalence between $(4)$ and $(5)$ (\textsl{i.e.} that measure hyperbolic implies general type), which is known to hold true whenever $X$ is a projective surface. In spite of this, in the sequel we will mostly concentrate ourselves on the conjecture \textsl{a latere}, known as Kobayashi's conjecture, which states that $(3)$ should imply $(5)$ (and hence $(8)$). We shall give some hints and recent results both from a differential-geometric and algebro-geometric viewpoints; moreover, we shall fix our attention on threefolds, since it is the first unknown case.

To begin with, observe that several powerful machineries from birational geometry ---such as the characterization of uniruledness in terms of negativity of the Kodaira dimension, the Iitaka fibration, the abundance conjecture (which is actually a theorem in dimension three)--- permit to reduce this conjecture to the following statement: a projective threefold $X$ of Kodaira dimension $\kappa(X)=0$ cannot be hyperbolic. By the Beauville-Bogomolov decomposition theorem and elementary properties of hyperbolic manifolds, in dimension three \emph{it suffices to show that a Calabi-Yau threefold is not hyperbolic}. Here, by a Calabi-Yau threefold we mean a simply connected compact projective threefold with trivial canonical class $K_X\simeq\mathcal O_X$ and $h^i(X,\mathcal O_X)=0$, $i=1,2$.

\subsection*{Differential-geometric viewpoint}

A weaker form of $(3)\Rightarrow(5)$ and $(8)$, more differential-geometric in flavor, is to show that negative holomorphic sectional curvature implies ampleness of the canonical bundle. This is known up to dimension three, by the work of \cite{sd_H-L-W10}. Again, the core here is to show that the negativity of the holomorphic sectional curvature of $(X,\omega)$ forces the real first Chern class $c_1(X)_\mathbb R$ to be non-zero. 

A possible proof of this fact goes as follows. First, observe that an averaging argument shows that negative holomorphic sectional curvature implies negative scalar curvature. Now, suppose that $c_1(X)_\mathbb R=0$. Then, there exists a smooth function $f\colon X\to\mathbb R$ such that $\operatorname{Ricci}(\omega)=i\partial\bar\partial f$. Now, take the traces with respect to $\omega$ of both sides: modulo non-zero multiplicative constants, on the left we find the scalar curvature and on the right the $\omega$-Laplacian of $f$ which must therefore be always non-zero. Hence, $f$ is constant and the scalar curvature must be zero, contradiction.

Unfortunately, having negative holomorphic sectional curvature is much stron\-ger than being hyperbolic. In \cite{sd_Dem97}, it is conjectured that a weaker notion of negative curvature, namely non-degenerate negative $k$-jet curvature, should be instead equivalent (and it is proved there that it actually implies hyperbolicity). Let us explain briefly what this notion is, referring to \cite{sd_Dem97} for more details.

Let $J_kX\to X$ be the holomorphic fiber bundle of $k$-jets of germs of holomorphic curves $\gamma\colon(\mathbb C,0)\to X$ and $J_kX^{\text{reg}}$ its subset of regular ones, \textsl{i.e.} such that $\gamma'(0)\ne 0$. There is a natural action of the group $\mathbb G_k$ of $k$-jets of biholomorphisms of $(\mathbb C,0)$ on $J_kX$, and the quotient $J_kX^{\text{reg}}/\mathbb G_k$ admits a nice geometric relative compactification $J_kX^{\text{reg}}/\mathbb G_k\hookrightarrow X_k$. Here, $X_k$ is a tower of projective bundles over $X$. In particular, it is naturally endowed with a tautological line bundle $\mathcal O_{X_k}(-1)$, as well as a holomorphic subbundle $V_k\subset T_{X_{k}}$ of its tangent bundle. 

\begin{definition}
The manifold $X$ is said to have \emph{non-degenerate negative $k$-jet curvature} if there exists a singular hermitian metric on $\mathcal O_{X_k}(-1)$ whose Chern curvature current is negative definite along $V_k$ and whose degeneration set is contained in $X_k\setminus(J_kX^{\text{reg}}/\mathbb G_k)$.
\end{definition}

Observe that if $X$ has negative holomorphic sectional curvature, then it naturally has a non-degenerate negative $1$-jet curvature. The following question seems therefore particularly appropriate.

\begin{question}
Is it true that if $X$ has non-degenerate negative $k$-jet curvature then $c_1(X)_\mathbb R\ne 0$?
\end{question}

\subsection*{Algebro-geometric viewpoint} 

Algebraic geometers expect more than the non hyperbolicity of Calabi-Yau's: a \textsl{folklore} conjecture states that every Calabi-Yau manifold should contain a rational curve. For threefolds, let us cite a couple of results in this direction:

\begin{itemize}
\item The article \cite{sd_HB-W92} is the culmination of a series of papers by Wilson in which he studies in a systematic way the geometry of Calabi-Yau threefolds; among many other things, it is shown there that if the Picard number $\rho(X)>13$, then there always exists a rational curve on $X$.
\item Following somehow the same circle of ideas, it was proven in \cite{sd_Pet91} (see also \cite{sd_Ogu93}) that a Calabi-Yau threefold $X$ has a rational curve provided there exists on $X$ a non-zero effective non-ample line bundle on $X$.
\end{itemize}

By the Cone Theorem, if there exists on a Calabi-Yau manifold $X$ a non-zero effective non-nef line bundle, then there exists on $X$ a rational curve (generating an extremal ray). Therefore, we can always suppose that such an effective line bundle is nef. Remark, on the other hand, that in Peternell's result, the effectivity hypothesis is crucial (regarding it in a more modern way) in order to make the machinery of the logMMP work. In this spirit, Oguiso asked in \cite{sd_Ogu93} the following question: \emph{is it true that if a Calabi-Yau threefold $X$ possesses a non-zero nef non-ample line bundle, then there exists a rational curve on $X$?}

Here is a positive answer, under a mild condition on the Picard number of $X$.

\begin{theorem}[Diverio, Ferretti \cite{sd_D-F}]
Let $X$ be a Calabi-Yau threefold and $L\to X$ a non-zero nef non-ample line bundle. Then, $X$ has a rational curve provided $\rho(X)>4$.
\end{theorem}

In order to give a rough idea of the techniques involved in this kind of business, let us state (a special case of) the Kawamata-Morrison Cone conjecture and explain how it would almost imply the Kobayashi conjecture.

\begin{conjecture}[Kawamata-Morrison]
Let $X$ be a Calabi-Yau manifold. Then, the action of $\operatorname{Aut}(X)$ on the nef-effective cone of $X$ has a rational polyhedral fundamental domain.
\end{conjecture}

\begin{proposition}
Suppose that the Kawamata-Morrison conjecture holds. Then, the Kobayashi conjecture is true in dimension three, except possibly if there exists a Calabi-Yau threefold of Picard number one which is hyperbolic.
\end{proposition}

\begin{proof}
We shall suppose that the Kawamata-Morrison conjecture holds true and that there exists a hyperbolic Calabi-Yau threefold $X$ with $\rho(X)\ge 2$ and derive a contradiction.


Since $X$ is supposed to be hyperbolic, it does not contain any rational curve and $\operatorname{Aut}(X)$ is finite. The Kawamata-Morrison conjecture implies therefore that the nef cone of $X$ is rational polyhedral.

Now, since it is rational polyhedral, rational points are dense on each face of the nef boundary. Moreover, at most one of these faces (which are at least in number of $\rho(X)\ge 2$) can be contained in the hyperplane given by $(c_2(X)\cdot D)=0$: in fact this is a \lq\lq true\rq\rq{} hyperplane since if $c_1(X)=c_2(X)=0$, then $X$ would be a finite \'etale quotient of a complex torus, so that $X$ would not be hyperbolic (nor a Calabi-Yau manifold in our strict sense). 

Therefore, there exists on $X$ a (in fact plenty of) nef $\mathbb Q$-divisor $D$ such that $c_2(X)\cdot D>0$. Computing its Euler characteristic and using Kawamata--Viehweg vanishing, $D$ can be shown to be effective. But then, since there exists on $X$ a non-zero effective non-ample divisor, there exists a rational curve on $X$, contradicting its hyperbolicity. 
\end{proof}


\begin{thebibliography}{99}

\bibitem{sd_Dem97}
J.~P.~Demailly; \textit{Algebraic criteria for Kobayashi hyperbolic projective varieties and jet differentials}, Algebraic geometry---Santa Cruz 1995, 285--360, 
Proc. Sympos. Pure Math., 62, Part 2, Amer. Math. Soc., Providence, RI, 1997.

\bibitem{sd_D-F}
S.~Diverio, A.~Ferretti; 
\textit{On a conjecture of Oguiso about rational curves on Calabi-Yau threefolds},
to appear on Comment. Math. Helv. 

\bibitem{sd_HB-W92}
D.~R.~Heath-Brown, P.~M.~H.~Wilson; \textit{Calabi-Yau threefolds with $\rho>13$}, 
Math. Ann. \textbf{294} (1992), no. 1, 49--57.

\bibitem{sd_H-L-W10}
G.~Heier, S.~Y.~Lu, B.~Wong; \textit{On the canonical line bundle and negative holomorphic sectional curvature}, 
Math. Res. Lett. \textbf{17} (2010), no. 6, 1101--1110.

\bibitem{sd_Ogu93}
K.~Oguiso;
\textit{On algebraic fiber space structures on a Calabi-Yau 3-fold},
Internat. J. Math. \textbf{4} (1993), no. 3, 439--465.

\bibitem{sd_Pet91}
T.~Peternell;
\textit{Calabi-Yau manifolds and a conjecture of Kobayashi},
Math. Z. \textbf{207} (1991), no. 2, 305--318.

\end{thebibliography}
\end{document}